\documentclass{article} 

\usepackage{amsmath,amsthm,amsfonts,amssymb}

\usepackage{upgreek}
\usepackage{wrapfig}

 \usepackage{geometry}
\pdfoutput=1  

\theoremstyle{definition}

\newtheoremstyle{example}{15pt}{20pt}%
     {}
     {}
     {\bfseries}
     {}
     {1pt}
     {\thmname{#1}\thmnumber{ #2.}~\thmnote{\textit{\textbf{#3}}}%
     \\[.15cm]\unskip\nobreak}

\theoremstyle{definition}

\newtheoremstyle{remark}{15pt}{20pt}%
     {}
     {}
     {\bfseries}
     {}
     {1pt}
     {\thmname{#1}\thmnumber{ #2.}~\thmnote{\textit{\textbf{#3}}}%
     \\[.15cm]\unskip\nobreak}

\newtheorem{example}{Example}[section]
\newtheorem{remark}{Remark}[section]
\newtheorem{definition}{Definition}[section]

\newtheorem{theorem}{Theorem}[section]

\def\Ex#1{Ex.~\ref{#1}}

\def\Re{\mathbb{R}}

\def\posRe{\mathbb{R}^+}

\def\fee{\upphi}

\def\clA{\mathcal{A}}
\def\clH{\mathcal{H}}

\makeatletter
\newcommand\gobblepars{%
    \@ifnextchar\par%
        {\expandafter\gobblepars\@gobble}%
{}}
\makeatother

\def\whamit#1{\smallbreak\pagebreak[3]%
	\noindent\textit{#1}\ \ \gobblepars}

\def\wham#1{\smallbreak\pagebreak[3]%
	\noindent\textbf{#1}\ \ \gobblepars}

\def\whamb{\wham{$\bullet$}}

\usepackage[usenames,dvipsnames]{color}
 \usepackage{url}            

\def\textsubscript#1%
{$_{\text{#1}}$}

\usepackage{empheq}

\def\tilh{\tilde h}
\def\tilz{\tilde z}
\def\tilp{\tilde p}
\def\Hil{\mathbf H}
\def\Ltwo{\mathbf{L}^2}
\newcommand{\Real}{\mathbb R}

\usepackage{eufrak}
\def\frakp{\mathfrak{p}}
\def\frakh{\mathfrak{h}}
\def\frakz{\mathfrak{z}}

\def\solp{\frakp}
\def\solh{\frakh}
\def\solz{\frakz}

\def\pop{\hbox{\tiny (POP)}}
\def\Kur{\hbox{\tiny (Kur)}}

\def\Theorem#1{Thm.~\ref{#1}}

\def\Sec#1{Sec.~\ref{#1}}
\def\Fig#1{Fig.~\ref{#1}}

\def\FRAC#1#2#3{\genfrac{}{}{}{#1}{#2}{#3}}

\def\half{{\mathchoice{\FRAC{1}{1}{2}}%
{\FRAC{2}{1}{2}}%
{\FRAC{3}{1}{2}}%
{\FRAC{4}{1}{2}}}}

\def\gen{{\cal D}}

\newcommand{\field}[1]{\mathbb{#1}}

\def\posRe{\field{R}_+}
\def\Re{\field{R}}
\def\Co{\field{C}}

\def\Sec#1{Sec.~\ref{#1}}

\def\eqdef{\mathrel{:=}}

\newcounter{rmnum}
\newenvironment{romannum}{\begin{list}{{\upshape (\roman{rmnum})}}{\usecounter{rmnum}
\setlength{\leftmargin}{6pt}
\setlength{\rightmargin}{4pt}
\setlength{\itemindent}{10pt}
}}{\end{list}}

\newcommand{\ud}{\,\mathrm{d}}

\newcommand{\barc}{\bar{c}}

\def\Expect{{\sf E}}
\newcommand{\avgT}{\lim_{T\to\infty}\frac{1}{T}\!\!\!\int_0^T}

\def\intsw{\int_{\Omega}\!\int_0^{2\pi}}

\newcommand{\Lscr}{\mathcal L}
\newcommand{\dtheta}{\theta\theta}
\newcommand{\ptheta}{\partial_{\theta}}
\newcommand{\pthetaB}{\partial^2_{\dtheta}}
\newcommand{\pt}{\partial_{t}}

\def\Prob{{\sf P}}

\def\Expect{{\sf E}}

\newcommand{\ith}{i^\text{th}}

\def\PointwiseErr{\Lscr^\alpha}
\def\wavePointwiseErr{{\PointwiseErr}}

\def\wavepProjC{P_c}
\def\wavepProjS{P_s}
\def\Kur{\hbox{\tiny \textrm{KUR}}}
\def\phz{\zeta}

\def\mfh{h}
\def\mfH{H}
\def\mfc{\bar{c}}

\def\Hopt{\underline{H}}

\def\waveh{h_0}

\def\wavec{\mfc_0}

\def\waveH{H_0}

\def\waveHopt{\Hopt_0}

\def\pop{\hbox{\tiny (POP)}}
\def\varble{\, \cdot\,}

\def\GalErr{e_{\hbox{\tiny Gal}}}

\def\waveGalErr{{\GalErr}}
\newcommand{\Noise}{\half\sigma^2}

\def\PtwoEqm{\bar{\alpha}_i}

\def\popamp{\wavepProjC^2 + \wavepProjS^2}

\newcommand{\pmat}[1]{\begin{pmatrix}#1 \end{pmatrix}}

\title{Functional role of synchronization: A mean-field control perspective\thanks{This research was supported by AFOSR under Grant No. FA9550-23-1-0060 and NSF under Grants   2336137 (PM)
and 2306023 (SM).}}

 \author{Prashant Mehta\thanks{University of Illinois at Urbana-Champaign, Urbana IL,	{\tt mehtapg@illinois.edu.}} 
 \and
 	Sean  Meyn\thanks{Department of ECE at the University of Florida, Gainesville FL,		{\tt meyn@ece.ufl.edu.}} 
}

\begin{document}
 	\maketitle

 \begin{abstract}%
The broad goal of the research surveyed in this article is to develop methods for understanding the aggregate behavior of interconnected dynamical systems, as found in mathematical physics, neuroscience, economics, power systems and  neural networks.   Questions concern prediction of emergent (often unanticipated) phenomena,   methods to formulate distributed control schemes to influence this behavior, and  these topics  prompt many 
other questions in the domain of learning.     The area of \textit{mean field games},   pioneered by Peter Caines,    are well suited to addressing these topics.    The approach is surveyed in the present paper within the     context of controlled coupled oscillators.   
\smallskip


\end{abstract}%

\begin{quote}
\textit{For our friend and mentor, Peter Caines, on his 80th birthday!}
\end{quote}

\section{Introduction}
\label{sec:intro}

Our friend and mentor Peter Caines has, together with his colleagues,  created new foundations for studying collective dynamics in complex  systems. Of particular inspiration to us has been his pioneering work in mean-field games (MFGs)  launched two decades ago \cite{cai21,huacaimal07,huang06large}, and the related field of mean-field control. 
Peter pointed the way to \textit{both} formulate and solve the problem of collective dynamics arising in a large population of heterogeneous dynamical systems. 
In this paper we survey some elements of MFGs within the context of controlled coupled oscillators.

We begin by introducing a model for a single oscillator: 
\begin{equation}
\ud \theta(t) = (\omega + u(t))\ud t + \sigma \ud \xi(t),\quad \text{mod} \;2\pi  
\label{e:osc}
\end{equation}
where
	$\theta(t) \in[0,2\pi)$ is the phase of the oscillator at time $t$,
$\omega$ is the nominal frequency with units of radians-per-second, $\{\xi(t):t\geq 0\}$ is a standard Wiener process, and $u(t)$ is a control signal whose interpretation depends on the context.  Unless otherwise noted,  the SDEs are interpreted in their It\^{o} form.

In this paper we consider a (possibly heterogeneous) population of phase oscillators,  in which the signal  $u_i(t)$ for the $i^{\text{th}}$ oscillator at time $t$ is determined by a combination of local and global information.  

A key insight from Peter's work is to formulate and analyze the collective dynamics of a large population as a solution of an optimal control problem in an infinite-population (mean-field) limit.  The transformation from a large population to its mean-field limit is also a theme in prior analysis of coupled oscillator models in the work of Strogatz~\cite{strogatz91stability,strogatz00from}.   Mean field games go beyond since the approach allows for incorporation of distributed decision making.  

After considering the most famous example of Kuramoto~\cite{kuramoto75international},  we describe the control-theoretic approach in \Sec{s:controlK}.

\subsection{Kuramoto coupled oscillator model and phase transition}

This celebrated example consists of  $N$ oscillators that evolve according to
\begin{equation}
\label{e:Kuramoto}
\ud \theta_i(t) = \left( \omega_i + \frac{\kappa}{N}\sum_{j=1}^N
\sin(\theta_j(t)-\theta_i(t))\right)\ud t  + \sigma \ud \xi_i(t),\quad\text{mod}\;2\pi,\quad
i=1,2,\hdots, N
\end{equation}
where $\theta_i(t)$ is the phase of the $i^{\text{th}}$-oscillator at time
$t$, $\omega_i$ is its natural frequency,  $\{\xi_i(t):t\geq 0,\;1\leq i\leq N\}$ are mutually independent standard Wiener processes, and $\kappa$ is the coupling parameter. 
In a heterogeneous
population, the frequency $\omega_i$ is 
sampled i.i.d. from a density $g(\omega)$ with support on $\Omega:=[1-\gamma,1+\gamma]$.  The parameters $\gamma$ and $\kappa$ are used to model the heterogeneity and the strength of network coupling, respectively.

\begin{figure*}[htp]
    \centering
   \includegraphics[width=0.35\hsize]{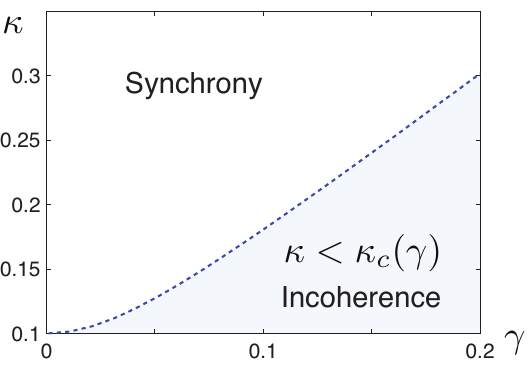}\vspace{-0.0in}
    \caption{Bifurcation diagram for the Kuramoto model.}
    \vspace{-0.in}
    \label{f:KuraBifur}
\end{figure*}

\textit{Phase transition}.
The dynamics of the Kuramoto model are visualized using a bifurcation
diagram in the $(\kappa,\gamma)$ plane, which in particular illustrates the
emergence of a phase transition. The stability boundary $\kappa =
\kappa_c(\gamma)$ shown in Fig.~\ref{f:KuraBifur} provides
an illustration of the phase transition: The oscillators behave incoherently
for $\kappa<\kappa_c(\gamma)$, and synchronize for $\kappa>\kappa_c(\gamma)$.
In the former incoherent population regime, the oscillators rotate close to their own
natural frequency and hence the trajectory $\{\theta_i(t):t\geq 0\}$ of the $i^{\text{th}}$-oscillator is approximately
independent of the population. In the synchronized regime, each oscillator
rotates with a common frequency.

While earlier studies from Kuramoto, Strogatz and others has been informative in understanding the phase transition from incoherence and synchrony, the relevance of such analysis to applications suffered from a fundamental problem: Why are these dynamics relevant to the functional role of the coupled oscillators?

Peter's research is helpful to systematically address this question.  Among the outcomes of the research program that he initiated, we point out two: (1) The functional role is encoded in the specification of an optimal control objective; and (2) an optimization formulation naturally yields algorithms for learning approximately optimal distributed control laws.

\subsection{A control-theoretic perspective on synchronization}
\label{s:controlK}

We introduce here a coupled oscillator model with finer statistical details.  The model requires specification of two densities: $p_0$ supported on $[0, 2\pi)$  
and $g$ supported on $\Omega$.     For a population of $N$ oscillators,   nominal frequency values $\{ \omega_i \}$  are i.i.d., with common marginal density $g$,    and initial phases $\{ \theta_i(0) \}$ are i.i.d.,  with common marginal density $p_0$.  For $t\ge 0$ the evolution of the $i^{\text{th}}$ phase vaiable is described as the SDE,   
\begin{equation}
\label{eq:theta_eqn_intro}
\ud \theta_i(t) = (\omega_i + u_i(t) )\ud t +  \sigma \ud \xi_i(t), \quad\text{mod}\;2\pi,  \quad \theta_i(0) \stackrel{\text{i.i.d}}{\sim} p_0, \qquad
i=1,2,\hdots, N,
\end{equation}
in which the standard Wiener processes $\{\xi_i(t):t\geq 0,\;1\leq i\leq N\}$ are mutually independent, and
$\{u_i(t):t\geq 0,\; 1\leq i\leq N\}$ are the control inputs. We let $p^{i,N}(\varble ,t)$ denote the marginal density for the $i^{\text{th}}$ oscillator at time $t$.

Kuramoto's model~\eqref{e:Kuramoto} is a particular case of~\eqref{eq:theta_eqn_intro} in which  
\begin{equation}
\label{e:Kuramoto_control}
u_i(t) = - \frac{\kappa}{N}\sum_{j=1}^N
\sin(\theta_i(t)-\theta_j(t))   \,, \qquad 1\le i\le N
\end{equation}
The right hand side is an example of a mean-field (or McKean-Vlasov) type control law~\cite{bensoussan2013mean}:   
 the control input $u_i(t) $  is a function of the local state $\theta_i(t)$ and the empirical distribution of the population.  
 As $N\to\infty $  the densities $p^{i,N}(\varble ,t)$  converge to a common density  $p(\varble,t)$, 
 which yields the mean-field approximation
 $u_i(t) \approx - \kappa \int_0^{2\pi} \sin(\theta_i(t) - \vartheta) p(\vartheta,t) \ud \vartheta$.

Going beyond the classical Kuramoto model,  the density $p(\varble,t)$  has two different interpretations that depend on the functional role:
\begin{enumerate}
\item \textbf{Control objective.} $p(\varble,t)$ has the meaning of a density where $u_i(t)$ is according to a mean-field type optimal control law. This requires a definition of an optimal control objective. In a distributed setting, a game formulation is natural. 
\item \textbf{Estimation objective.} $p(\varble,t)$ has the meaning of a posterior density in nonlinear filtering. This requires a definition of a hidden Markov process. 
\end{enumerate}

Apart from control and estimation, the third objective is that of learning. 
The \textbf{learning objective} is to learn an approximate control law from some parameterized family.  In optimal control setting, such an objective is referred to as policy optimization.
The goal is to design learning rules to learn the optimal parameters. These rules are designed such that, upon completion of learning, the parametrized control law gives a good approximation of the optimal control law.

In this paper, the three objectives are illustrated through a discussion of our prior work: \Sec{sec:mfg} concerns control,~\Sec{sec:learning} is on learning, and~\Sec{sec:FPF} is on estimation.  Before delving in to the details, we comment briefly upon the importance of phase transitions and synchronization in applications.

\subsection{Emergence and phase transition in applications}
\label{s:Melania}

\textit{Emergence} is an important concept in fields ranging from  mathematical physics, neuroscience, economics to AI. 
An example that captures the interest of the  popular media is  the emergent and unexpected capabilities of large language models (LLMs): At a certain scale of computation and model size, there is a phase transition whereby a novel (and often unexpected) capability \textit{emerges} that is not present in smaller models~\cite{wei2022emergent,geshkovski2023mathematical}.    

From our own prior work, an important example is the research found in power systems, in which a vast population of electric loads under distributed control might synchronize without careful design, destabilizing the power grid \cite{meybarbusyueehr15}.  We describe this example followed by a discussion of the coupled oscillator as models of oscillatory activity in cortical networks.

\wham{Intelligent appliances.}

Since the 1970s it has been recognized that demand-side flexibility of electric loads can help to maintain supply-demand balance in the power grid.  The value is far greater today, since volatile energy from the sun and wind increases the need for balancing resources.  One goal of the the Energy Policy Act of 2005 was to accelerate the adoption of demand-side resources \cite{qdr2006benefits}.

For the grid operator or utility company, the desired services include the following:
\whamb  Peak shaving and valley filling.
\whamb Resources that can ramp up (or down) quickly due to a surge in demand or supply of electricity.
\whamb   There is high frequency volatility that is managed today by generators as part of  automatic generation control.

The term ``virtual energy storage'' (VES) is used to describe these services obtained from electric loads rather than generation or traditional storage devices.  
 Since 2005 there has been a resurgence of interest within the academic community:  see the surveys included in  \cite{IMA18}, in particular \cite{cheche17b,almesphinfropauami18,chehasmatbusmey18,moymey17a}.   Notable articles and theses include   \cite{johThesis12,linbarmeymid15,YueChenThesis16,NeilCammardellaThesis21,JoelMathiasThesis21,bencolmal19,lennazmal21}.   

 Consider the special case of residential water heaters, which are inherently energy storage devices---they store hot water!      As long as the water temperature stays within desired bounds, the energy consumption is highly flexible.    The objective of   load control is to provide grid services while (1) ensuring constraints on quality of service are maintained for all time, and (2) minimizing the chance of  synchronization.

 The heating-cooling temperature dynamics in a single residential water heater are naturally modeled as a single oscillator \cite{busmeycam23} when water temperature control is carried out based on a hysteresis policy.  It is possible to modify this policy so that it takes    into account a reference signal from the grid operator, so that aggregate power consumption from the ensemble of water heaters tracks a reference with high accuracy.  It is crucial that each load has available some global information to ensure reliable tracking, and to minimize the probability of the worst kind of emergent behavior:  synchronization of the population of loads, leading to massive surges in power consumption.  A range of approaches may be found in the aforementioned work;  in particular, \cite{lennazmal21} employs the MFG approach much like in the present paper.

 \wham{Neuroscience.}  The 
mammalian auditory cortex is a laminar structure organized across six
layers with defined input and output layers that support
oscillatory rhythms across unique frequency
bands~\cite{winkowski2013laminar,king2018recent,moerel2014anatomical}. While rhythms
are ubiquitous---conserved across every nervous system from fruit flies
to humans---their functional role is one of the most fundamental
mysteries in all of neuroscience~\cite{buzsaki2013scaling}. 
 
There is well-established normal form reduction  procedure to approximate the behavior of a single neuron by an oscillator model~\cite{Gucken_75,Brown_Moehlis_04,Monga2018}. An oscillator model captures two aspects of the dynamics: (1) The tonic firing of a neuron; and (2) The effect of the input $u(t)$ on the phase (this requires also an additional concept of a phase response curve).  The implicit assumption is that the input $u(t)$ affects only the phase of the tonic firing.  The assumption is justified if the limit cycle is strongly attractive (i.e., Floquet exponents have large negative real parts) and $u(\cdot)$ is small.  In a network, the input to a neuron arises from a cumulative effect of the population and the assumption is referred to as a weakly coupled network~\cite{Ermentrout}. See~\cite{bick2020understanding} for a recent review on the use of \textit{coupled} oscillator models in neuroscience.

\subsection{Relevant literature on oscillator models and synchronization}

Since the model was introduced in the 
1970s~\cite{kuramoto75international} several extensions of the basic
model~\cite{kalloniatis2017synchronisation,Barreto2008,Murphy2010} and
analytical
techniques~\cite{strogatz91stability,strogatz00from,Ott2008,Martens2009,smith2020finite}
for its solution have
been developed.  During the past few decades, there have been notable
contributions related to \textbf{(1)} extension of the normal form
reduction procedure to capture the effects of
amplitude coordinate~\cite{wedgwood2013phase,shirasaka2017phase,Wilson2020}; \textbf{(2)}
study of synchronization and entrainment in the presence of an
additive external (typically harmonic)
input~\cite{OttJetlag,Sakaguchi1988,Kori2006,Antonsen2008,Popovych2011,Childs2008,Hodson2021};
and \textbf{(3)} control of
synchronization~\cite{WILSON2022327,MONGA2019115,dorfler2014synchronization,meybarbusyueehr15,taydhocal16,nikcanfra22,busmeycam23}.

Much of this paper is a survey of mean field games and related topics in learning for control systems models inspired by the Kuramoto model.  The survey is largely drawn from our prior work on these topics~\cite{yin2011synchronization,yin2014efficiency,huibing_TAC14,taoyang_TAC12,variational_FPF,taghvaei2023survey}.  The novelty comes from the presentation of these disparate works as part of the same overarching theme which is very much inspired by Peter's work.   

\subsection{Paper outline}

The remainder of this paper is organized as follows.  The MFG formulation for the coupled oscillators is described in~\Sec{sec:mfg}.  Its application to development of learning algorithms appears in~\Sec{sec:learning}. \Sec{sec:FPF} concerns nonlinear filtering where the coupled oscillator feedback particle filter (FPF) algorithm is discussed. \Sec{sec:conc} contains some conclusions.

\section{Mean-field oscillator game}
\label{sec:mfg}

\subsection{Problem formulation}

Consider the dynamic game in which the  $i^{\text{th}}$ oscillator, subject to dynamics~\eqref{eq:theta_eqn_intro}, minimizes  
\begin{align}
    \eta_i^{\pop}(u_i;u_{-i}) = \limsup_{T\to\infty}\frac{1}{T}
    \int_0^T  \left( \frac{1}{N} \sum_{j=1}^N c^\bullet(\theta_i,\theta_j) +\half  R u_i^2 \right)\ud s 
     \label{eqn:obj}
\end{align}
where   $u_{-i}= (u_j)_{j\neq i}$.
A Nash equilibrium in control
policies is given by $\{u_i^*:1\leq i\leq N\}$ such that
$u_i^*$ minimizes $\eta_i^{\pop}(u_i;u_{-i}^*)$ for $i = 1, \hdots, N$.

The continuous and non-negative cost function   $c^{\bullet}:[0,2\pi)\times [0,2\pi)\mapsto \posRe$  and the control
penalty parameter $R>0$ are assumed to be
common to the entire population.

The following additional assumptions are imposed on the cost:   
 
\whamit{Spatially invariance:}
		$c^{\bullet}(\vartheta,\theta)=c^{\bullet}(\vartheta-\theta)$
		
\whamit{Even:} 
		$c^{\bullet}(\theta) =  c^{\bullet}(-\theta)$.
  Consequently, it admits the  Fourier series expansion 
\begin{align*}
c^\bullet(\theta) &= C_0^\bullet + \sum_{k=1}^{\infty}
C_k^\bullet\cos(k \theta),\qquad \theta\in[0,2\pi)
\label{eqn:Fourier_for_cost}
\end{align*}

We will see that the special case described next (\Ex{ex:cost_kura})  is useful to relate dynamic game defined here to the Kuramoto model:

\begin{subequations}   

\begin{example} 
\label{ex:cost_kura}
The spatially invariant and even cost function defined by
\begin{align}
 c^\bullet(\vartheta,\theta) &= \tfrac{1}{2}
\sin^2\left(\frac{\vartheta-\theta}{2}\right)
\\
& \text{for which} \quad C_0^\bullet = \tfrac{1}{4} \, , \  C_1^\bullet = - \tfrac{1}{4} \,, \ C_k^\bullet = 0 \ \textit{for } \  k \ge 2 
\end{align} 
 \end{example}
 
\end{subequations}

\begin{remark}
The MFG problem for coupled oscillators, the model~\eqref{eqn:obj}, was introduced by our group in \cite{yin2011synchronization}.  For extensions of this basic model and related papers on synchronization in MFG, see~\cite{carmona2020jet,carmona2023synchronization,cesaroni2024stationary,soner2024viscosity,hofer2024synchronization}. A closely related work is the analysis of the Cucker-Smale model~\cite{nourian2011mean}. 
\end{remark}

\subsection{Forward-backward MFG partial differential equation (PDE)}

We describe here the mean field limit as $N\to\infty$,   defined by  the limiting density $p=p(\theta,t;\omega)$ and a \textit{relative value function}  $h=h(\theta,t;\omega)$. The associated optimal control law in the mean field game is given by 
\begin{align}
\fee(\theta,t;\omega) := -\frac{1}{R}\ptheta h(\theta,t;\omega)
 \label{eqn:optuPDE}
\end{align}
A candidate approximate solution to the dynamic game \eqref{eq:theta_eqn_intro} for finite $N$ is $u_i(t) =  \fee(\theta_i(t),t;\omega_i)$.

These pair $(h,p)$ solve the  forward-backward PDE  
\begin{subequations}
\label{eq:MFG_FB}
\begin{align}
\pt h + \omega\ptheta h  &= \frac{1}{2R}(\ptheta h)^2 -
      \barc(\theta,t) + \eta^* -
    \frac{\sigma^2}{2}\pthetaB h
 \label{eqn:HJB}\\
 \pt p + \omega\ptheta p  &= \frac{1}{R}\ptheta\left[p(\ptheta h)\right] +
    \frac{\sigma^2}{2}\partial^2_{\theta\theta}p
\label{eqn:p}\\
    \barc(\vartheta,t) &= \intsw c^\bullet(\vartheta,\theta)     p(\theta, t;\omega) g(\omega) \ud\theta \ud\omega
\label{eqn:barc}
\end{align}
Equation~\eqref{eqn:HJB} is the Hamilton-Jacobi-Bellman (HJB) equation that describes the solution of minimizing~(\ref{eqn:obj}) under the assumption of a known deterministic mass influence $\bar{c}(\theta,t)$.  
Equation~\eqref{eqn:p} is the Fokker-Planck-Kolmogorov (FPK) equation that describes the evolution of the density with the optimal control input defined according to~\eqref{eqn:optuPDE}. The two PDEs are coupled via~\eqref{eqn:barc}. 

\label{eq:p+h}
\end{subequations}

The \textit{incoherence solution} is defined by
\begin{equation}
 h(\theta,t;\omega)  = \solh_0(\theta)  \eqdef 0
 \qquad
 p(\theta;t,\omega) = \solp_0(\theta)  \eqdef
\frac{1}{2\pi}
\label{e:incoherencesolution}
\end{equation}
Just as in the Kuramoto oscillators literature,  incoherence  means that the phase values $\{\theta_i(t):i=1,2,\hdots,N\}$  are uniformly distributed on the circle  $[0,2\pi]$~\cite{strogatz91stability}. 

Subject to \eqref{e:incoherencesolution}, the control law \eqref{eqn:optuPDE} results in  $\fee(\theta,t;\omega) \equiv 0$, and the cost $\barc$
defined in \eqref{eqn:barc} is independent of $(\vartheta, t)$.

\begin{example}[Continued from \Ex{ex:cost_kura}]\label{ex:cost_kura_2}
With $c^\bullet(\vartheta,\theta) =  \frac{1}{2} \sin^2\left( \frac{\vartheta -     \theta}{2}\right)$,
\[
    \barc(\vartheta,t) =  \frac{1}{2} \frac{1}{2\pi} \intsw \sin^2\left( \frac{\vartheta -
    \theta}{2}\right)g(\omega) \ud\theta \ud\omega
    = \frac{1}{4}
\]
which coincides with the average cost   $\eta^*(\omega) = \eta_0 \eqdef  \barc
$ for all $\omega\in\Omega$.    This value  is approximately consistent with
the finite-$N$ model.   When each control is set to zero we obtain
$\ud\theta_i(t) = \omega_i \ud t   + \sigma\ud\xi_i(t)$ for each $i$, which
results in average cost independent of $i$,
\[
    \avgT   c(\theta_i(t);\theta_{-i}(t))  \ud t =  \frac{N-1}{N}    \eta_0
    \]
There is a trade-off between reducing the cost associated with  $\theta_i(t) \neq
\theta_j(t)$, and reducing the cost of control. These competing costs
suggest a potential phase transition as the parameter $R$ varies from large ($R=\infty $) to small ($R=0+$). 
\end{example}
In order to investigate the phase transition, a linear stability analysis is carried out.

\subsection{Linear stability analysis}

The linearization of the equations~(\ref{eqn:HJB}) - (\ref{eqn:barc}) is taken
about the equilibrium incoherence solution $\solz_0 =(\solh_0,\solp_0)$.  A
perturbation of this solution is denoted $\solz_0 + \tilz = (\solh_0 ,\solp_0
)+(\tilh, \tilp)$, where $\tilz(\theta,t;\omega)=\tilz(\theta+2\pi,t,\omega)$
for all $t\in\posRe $ and $\omega\in\Omega$. Since $p = \solp_0+\tilp$
is a probability density, the perturbation satisfies the normalization
condition $\int_0^{2\pi}\tilp(\theta, t;\omega)\ud\theta = 0$  for all
$t,\omega$.    We impose  a similar normalization condition for the relative value function:
$\int_0^{2\pi}\tilh(\theta,t;\omega)\ud\theta = 0$  for all $t,\omega$,  which is  justified because $h$  is only defined up  to a constant.

When $\tilz$ is small, its evolution is approximated by the linear equation,
\begin{subequations}
\begin{equation}
    \frac{\partial}{\partial t} \tilz(\theta,t;\omega) =\Lscr_R\tilz(\theta,t;\omega)
    \label{eqn:lin_PDE}
\end{equation}
where
\begin{equation*}
    \Lscr_R\tilz(\theta,t;\omega) \eqdef
    \pmat{-\omega\ptheta \tilh - \frac{\sigma^2}{2}\pthetaB \tilh
    \\
    -\omega\ptheta \tilp + \frac{1}{2\pi R}\pthetaB \tilh + \frac{\sigma^2}{2}
    \pthetaB \tilp} -
    \pmat{ \tilde{c}(\theta,t)
    \\
    0}
\end{equation*}
and
\begin{align*}
    \tilde{c}(\theta,t)  &=
     \int_{\Omega}\int_0^{2\pi}c^\bullet(\theta,\vartheta) \tilp(\vartheta, t; \omega)
    g(\omega) \ud \vartheta\ud\omega.
\end{align*}

The local analysis entails well-posedness (existence, uniqueness) and
stability with respect to an infinitesimal initial perturbation of the
population density:
\begin{equation}
\tilde{p}(\theta,0;\omega) = q(\theta,\omega)
\label{eqn:p0_init_condn}
\end{equation}\label{eqn:linear_IVP}
\end{subequations}
where $\int_0^{2\pi}q(\theta, \omega)\ud\theta = 0$.

The analysis requires the introduction of a Hilbert space, taken here to be
$\Ltwo(\Real^+,\Hil\times\Hil)$, where $\Hil$ is a subspace of 
$\Ltwo([0,2\pi]\times\Omega)$.  The space
$\Ltwo([0,2\pi]\times\Omega)$ is defined with respect to the measure
$g(\omega) \ud \omega \ud\theta$ on the product space $[0,2\pi] \times
\Omega$.
For any complex-valued function $v(\theta,\omega)$ on $[0,2\pi] \times
\Omega$ we denote,
\[
\| v\|^2_{\Hil} \eqdef
     \int_0^{2\pi} \int_{\Omega} |  v(\theta,\omega)
     |^2  g(\omega) \ud \omega \ud\theta
\]
The Hilbert space $\Hil$ is defined to be the set of functions for which the
integral is finite, and $\int_0^{2\pi} v(\theta,\omega) \ud\theta = 0$ for all
$\omega\in\Omega$.  
We denote $\Hil^{2}:=\Hil\times\Hil$.

We refer to \eqref{eqn:linear_IVP} as the
linear initial value problem,  which is well-posed if a unique
solution $(\tilde h,\tilde p)(\theta,t;\omega)$ exists in $\Ltwo(\Real^+,\Hil^{2})$ for any initial
perturbation $\tilde p(\theta, 0, \omega) = q(\theta,\omega)\in \Hil$.  Along
with well-posedness, we are interested in local stability of the incoherence
solution:

\begin{definition}
    Consider the incoherence solution $\solz_0=(\solh_0, \solp_0)$ of the
    coupled nonlinear PDE~\eqref{eq:MFG_FB}.  The incoherence
    solution $\solz_0=(\solh_0,\solp_0)$ is linearly asymptotically stable if a solution $\tilp(\theta, t; \omega)$ of the linear initial 
    value problem~\eqref{eqn:linear_IVP}
    with an arbitrary initial perturbation $\tilp(\theta, 0; \omega) =
    q(\theta,\omega)\in \Hil $ exists in $\Ltwo(\Real^+,\Hil^2)$, 
    and satisfies $\|\tilp(\theta,t;\omega)\|_{\Hil} \rightarrow 0 $ as
    $t\rightarrow \infty$.   
\end{definition}

For the stability analysis of the linear initial value
problem~\eqref{eqn:linear_IVP}, it is useful to first
deduce the spectra.   For each non-zero integer $k$, taking either positive or negative values, denote
\begin{equation}
S_c^{(k)} \eqdef  \Bigl\{\lambda\in\Co\, \big| \, \lambda =     \pm\frac{\sigma^2}{2}k^2 - k\omega i \text{ for all }
    \omega\in\Omega\Bigr\}
\label{e:ctsSpec}
\end{equation}
The following is taken from  Thm.~~4.1 of \cite{yin2011synchronization}, whose proof may be found in~\cite[Appendix~VII-D]{yin2011synchronization}.

\begin{theorem}
\label{t:Discrete_and_continuous}
For the linear operator $\Lscr_R:\Hil^{2}\rightarrow \Hil^{2}$,
\begin{romannum}
\item The continuous spectrum is the union of sets $\{S_c^{(k)}  :  |k| \neq 0\}$.
 
\item The discrete spectrum equals the union of sets $\{S^{(k)}_d  :  |k|\neq 0\}$,  where
    \begin{align}
	S^{(k)}_d \eqdef \Bigl\{\lambda\in\Co\ \big | \    C_{|k|}^\bullet  \frac{k^2}{2R}  
	\int_{\Omega}
        \frac{g(\omega)}{(\lambda -\frac{\sigma^2}{2}k^2 + k\omega i)(\lambda +
	\frac{\sigma^2}{2}k^2 + k \omega i)}\ud\omega - 1 = 0
	\Bigr\}
\label{s:Skd}
    \end{align}
 Observe that 
$S^{(k)}_d = \emptyset$  {\rm (the empty set)} if $C_{|k|}^{\bullet}=0$.  

\end{romannum}
\end{theorem}

\medskip

The points in $S_c^{(k)}$ are in one-one correspondence with the frequencies
in the support of the distribution $g(\omega)$.  That is, for each
$\omega_0\in\Omega$, the point $\pm\frac{\sigma^2}{2}k^2 - k\omega_0 i\in
S_c^{(k)}$ lies in the continuous spectrum. On the complex plane, $S_c^{(k)}$
comprises of two line segments, one in the left half-plane and the other in
the right half-plane.  The main thing to note is that the
continuous spectrum does not change with the value of $R$ and is moreover
bounded away from the imaginary axis for $k=\pm 1,\pm 2, \cdots$.  So, the
focus of the analysis is on the discrete
spectrum.  As follows from Thm.~\ref{t:Discrete_and_continuous}, the discrete
spectrum is obtained by solving the characteristic equation---the equality in \eqref{s:Skd}---for $k= 1, 2,\cdots$,  subject to $C_k^{\bullet} \ne 0$ so that the equation has a feasible solution. For negative values of $k$,
the eigenvalues are simply the complex conjugate.

\medskip

\begin{example}[Continued from \Ex{ex:cost_kura} and~\ref{ex:cost_kura_2}]
Consider $c^\bullet(\vartheta-\theta) =
\half\sin^2\left(\half [ \vartheta-\theta]\right)$.  In this case
$C_1^{\bullet}=-\tfrac{1}{4}$ and there is only one characteristic equation to
consider
\begin{align*}
\frac{1}{8R} \int_{\Omega}
        \frac{g(\omega)}{(\lambda -\frac{\sigma^2}{2} + \omega i)(\lambda +
	\frac{\sigma^2}{2} +  \omega i)}\ud\omega + 1 = 0
\label{eqn:chara}
\end{align*}
This equation was solved numerically to obtain
a path of eigenvalues as a function of $R$. In these calculations, 
we fixed $\sigma^2=0.1$ and
$g(\omega)$ the uniform distribution on $\Omega = [1-\gamma,1+\gamma]$.  

\begin{figure*}[ht]
    \centering
\includegraphics[width=\hsize]{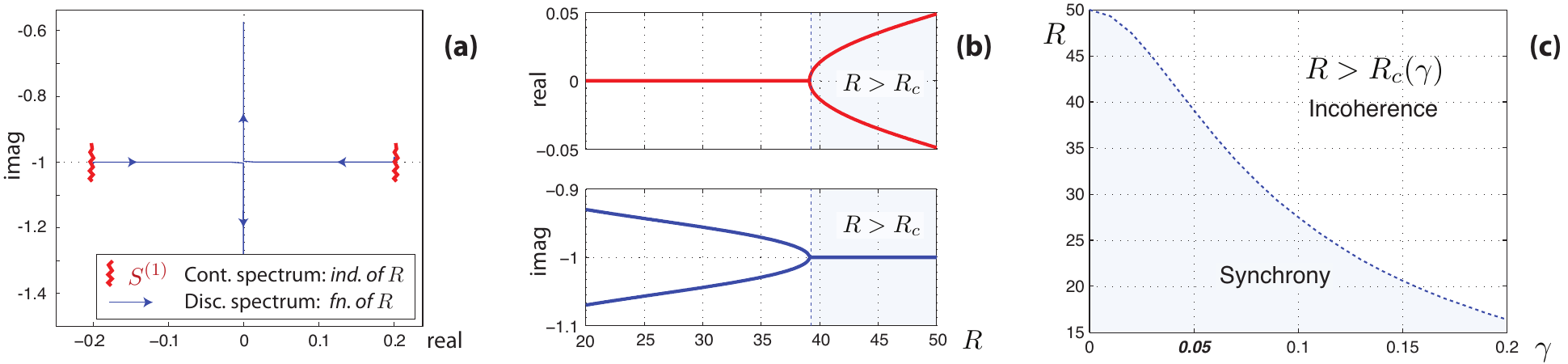}
    \caption{Spectrum as a function of $R$. (a) The continuous spectrum for
    $k = 1$, along with the two eigenvalue paths as $R$ decreases. (b) The
    real and imaginary parts of the two eigenvalue paths as $R$ decreases.
    (c) $R_c(\gamma)$ as a function of $\gamma$.}
    \label{fig:chara}
\end{figure*}

\Fig{fig:chara}~(a) depicts the resulting locus of eigenvalues
obtained with $\gamma=0.05$.  For $R\sim
\infty$ there are  a pair of complex eigenvalues at $\sim\pm\frac{\sigma^2}{2} -
i$.   As the parameter $R$ decreases,  these eigenvalues move
continuously towards the   imaginary axis.   The critical value $R_c$ is
defined as the value of $R$ at which these two eigenvalue paths collide on the
imaginary axis,  resulting in an eigenvalue pair of multiplicity $2$.   The
eigenvalues  split as $R$ is decreased further,  and remain on the imaginary
axis for $R < R_c$.  The real and the imaginary part of the two eigenvalue
paths originating at $\pm \frac{\sigma^2}{2} - i$ are depicted in
\Fig{fig:chara} (b). These eigenvalues also have their complex conjugate
counterparts (for $k=-1$) that are not depicted for the sake of clarity.

In (a) and (b) the value of $\gamma$ is fixed at $0.05$.  The critical value
$R_c$ is a function of the parameter $\gamma$.  \Fig{fig:chara} (c) depicts a
plot of $R_c(\gamma)$ as a function of $\gamma$.  For the uniform distribution
$g(\omega)$, the critical point also has an analytical expression:
\begin{align*}
	R_c(\gamma) = \begin{cases}
		\frac{1}{2 \sigma^4} & \text{ if }\, \gamma = 0,\\
		\frac{1}{4 \sigma^2
		\gamma}\tan^{-1}\left(\frac{2\gamma}{\sigma^2}\right)& \text{
		if }\, \gamma > 0.
	    \end{cases} 
\end{align*}
This formula is consistent with the expression for
\textit{critical coupling} for the
Kuramoto model in~\cite{strogatz91stability},  which defines the critical value for $\kappa$ in \eqref{e:Kuramoto}
by
$\kappa_c(\gamma) =
\frac{2\gamma}{\tan^{-1}\left( 2\gamma\sigma^{-2}\right)}$). 
\end{example}

For values of $R$ larger than the critical threshold $R_c(\gamma)$, we have the following stability conclusion for the linear initial value problem. 

\begin{theorem}[Thm.~~4.3 in~\cite{yin2011synchronization}]
Consider the linear initial value
problem~\eqref{eqn:linear_IVP} with
$c^\bullet(\vartheta-\theta)=
\half\sin^2\left(\frac{\vartheta-\theta}{2}\right)$.
For $R>R_c(\gamma)$,
\begin{romannum}
    \item{\bf Existence and uniqueness.} A unique solution $(\tilde{p}, \tilde{h}) \in \Ltwo(\Real^+,\Hil^2)$ exists. 
    \item{\bf Asymptotic stability of incoherence solution.} As $t\rightarrow\infty$, $\tilde p(\theta,t;\omega)
	\rightarrow 0$.
\end{romannum}
\end{theorem}
\begin{proof}
See Appendix~VII-F in~\cite{yin2011synchronization}.
\end{proof}

\subsection{Bifurcation and phase transition}
\label{ssec:bif+phasetrans}

In~\cite{yin2011synchronization,yin2011bifurcation}, bifurcation results are established for values of $R$ below the critical threshold $R_c(\gamma)$. Specifically, for the homogeneous population ($\gamma=0$), it is shown that below the critical threshold $R=R_c(\gamma)$ bifurcates a branch of traveling wave solutions
\begin{equation}
(p(\theta,t;\omega),h(\theta,t;\omega))=(v^2(\theta- a t),- \half {\sigma^2 R} \ln
v^2(\theta- a t)),\quad R < R_c(\gamma) 
\label{e:trwave}
\end{equation}
with wave-speed $a=1$ (see~\cite[Theorem 4.6, Corollary 4.1 and Remark 4]{yin2011synchronization} where additional details on the function $v$ are provided). 

An explanation for the wave-speed $a=1$ is that the frequencies $\omega^i$ in a homogeneous population setting are identical equal to $1$ when $\gamma=0$.
 A similar bifurcation result is also possible for the heterogeneous population.  In this case, the wave-speed $a$ depends upon $g(\omega)$~\cite{yin2011bifurcation}. For   ease of exposition, we focus discussion on the homogeneous population with $a=1$.

Based on the bifurcation analysis, we have identified the following two types of solutions for the PDE model~\eqref{eq:MFG_FB}: 
\begin{romannum}
    \item
    \textbf{Incoherence solution} (For $R>R_c(\gamma)$):
	\begin{align*}
	    p(\theta,t;\omega) = \frac{1}{2\pi},\quad
	    h(\theta,t;\omega) = 0\, 
	\end{align*}
    \item \textbf{Synchrony solution} (For $R<R_c(\gamma)$):  The traveling wave equation,
	\begin{align*}
	    \!\!\!\!\!\!\!
	   p(\theta,t;\omega) = p(\theta-at,0;\omega),\quad 
	   h(\theta,t;\omega) = h(\theta-at,0;\omega)
	\end{align*}
\end{romannum}

In the following, with a slight abuse of notation, $p(\theta;\omega)$ is
used to denote $p(\theta,0;\omega)$,   
 and similarly
$h(\theta;\omega):=h(\theta,0;\omega)$. 
The traveling wave
solution is obtained simply by rotating this solution with a constant
wave speed, $p(\theta,t;\omega)=p(\theta-at;\omega)$ and
$h(\theta,t;\omega)=h(\theta-at;\omega)$.

\subsection{Comparison with Kuramoto}
\label{ssec:comp_with_Kuramoto}

A comparison of the phase transition diagram with the Kuramoto and the MFG models (with cost
$c^\bullet(\vartheta,\theta) = \frac{1}{2} \sin^2 (\frac{\vartheta-\theta}{2})$
in \Ex{ex:cost_kura}) is depicted in \Fig{f:KuraPDEBifur}. The plot shown on the right in \Fig{f:KuraPDEBifur} depicts the phase transition for the MFG model: For $R>R_c$, the oscillators are incoherent, and for $R<R_c$
the oscillators synchronize. That is,   \textit{synchronization occurs when the
control is sufficiently cheap}.

\begin{figure*}[htp]
    \centering
    \includegraphics[scale=0.65]{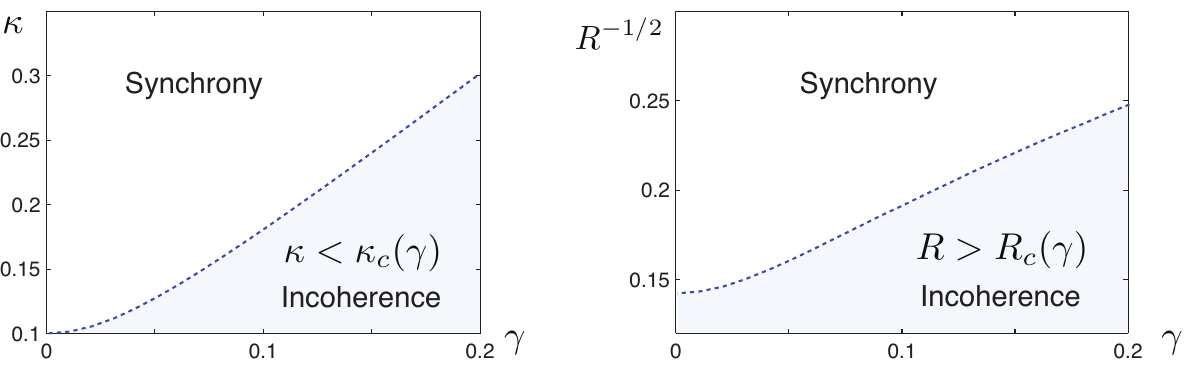}\vspace{-0.0in}
    \caption{Bifurcation diagrams. The Kuramoto model \eqref{e:Kuramoto} with $\sigma^2/2 = 0.05$
    (left), and the coupled model considered in this paper with $\sigma^2/2 =     0.05$ (right).}
    \vspace{-0.in}
    \label{f:KuraPDEBifur}
\end{figure*}

Recall that the MFG optimal control law is given by $u^* = -\frac{1}{R}\ptheta h(\theta,t;\omega)$. For the MFG solution, \Fig{fig:ControlComp} depicts the optimal control in relation to the density $p(\theta,t)=\int_\Omega p(\vartheta,t;\omega) g(\omega) \ud \omega$. Also depicted is the Kuramoto control ${u}^{\Kur}(\theta,t) := - \kappa \intsw \sin(\theta - \vartheta)
p(\vartheta,t;\omega) g(\omega)  \, \ud \omega\, \ud\vartheta$ (mean-field approximation based on formula~\eqref{e:Kuramoto_control}). The MFG solution is obtained via a numerical solution of the MF-HJB equation (see~\cite[Sec.~V]{yin2011synchronization}).   An appropriate value of the
parameter $\kappa$ is chosen to provide the comparison with the Kuramoto. 

The fact that the Kuramoto control is qualitatively similar to the optimal control provides a motivation for developing learning algorithms based on the parameterized form of the Kuramoto control law.  This is the subject of the next section.

\begin{figure}
    \centering
    \includegraphics[width=0.5\hsize]{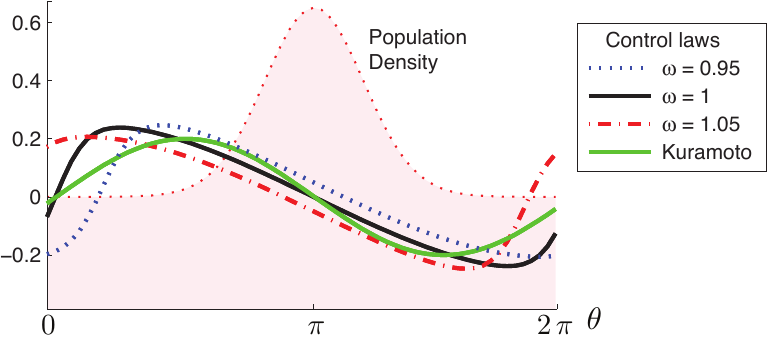}
    \caption{Comparison of the control obtained from solving the
      mean-field game PDE model and
    from Kuramoto model.}\label{fig:ControlComp}
\end{figure}

\section{Learning algorithms}
\label{sec:learning}

Consider the MFG optimal control problem~\eqref{eqn:obj} for the $i^{\text{th}}$ oscillator subject to the dynamics~\eqref{eq:theta_eqn_intro} with   cost $c^\bullet(\vartheta,\theta) = \frac{1}{2}
\sin^2\left(\frac{\vartheta-\theta}{2}\right)$ (recall \Ex{ex:cost_kura}). In principle, the forward-backward PDE~\eqref{eq:MFG_FB} can be solved numerically to obtain the optimal control law for the $i^{\text{th}}$ oscillator. 
  In practice, it is more reasonable to learn the best policy within some parameterized family.  
  
  In this section, taking inspiration from~\Fig{fig:ControlComp}, we consider 
\begin{equation}
\label{e:Kuramoto_control_parameter}
u_i(t)  = - \frac{1}{R} A_i(t) \frac{1}{N} \sum_{j=1}^N 
    \sin(\theta_i(t) - \theta_j(t) - \phz_i(t))
\end{equation}
    where $\{(A_i(t),\phz_i(t)):t\geq 0\}$ are the amplitude and phase parameters (compare with the Kuramoto control~\eqref{e:Kuramoto_control}).  The goal is to design learning rules to learn these parameters. The theory presented below is applicable to more general types of parameterizations~\cite{huibing_TAC14}.

\subsection{Mean-field approximate dynamic programming}
\label{sec:mf-adp}

Let $ \{ h^\alpha :  \alpha \in\clA  \subset  \Re^d\}$ denote a family of functions intended to approximate the relative value function appearing in \eqref{eq:p+h}.   
We describe here methods to choose the best parameter $\alpha^* \in \clA$,   followed by heuristics that led us to a specific choice for the function class.  

Following standard practice in reinforcement learning we introduce a version of the Q-function,
\begin{equation}
    \mfH^\alpha(\theta,t,u;\omega,p)  = \mfc(\theta,t;p) + \half Ru^2 +
    \gen_{u} \mfh^\alpha(\theta, t;\omega,p),\quad \theta\in[0,2\pi),\;\;t\geq 0, \;\; u\in\Re
  \label{e:mf_Qgen}
\end{equation}
where $\gen_u$ is the generator of the SDE~\eqref{eq:theta_eqn_intro},
and denote $\clH = \{ \mfH^\alpha :  \alpha\in\clA \}$.    
For any $\alpha\in\clA$ we obtain a policy motivated by 
 \eqref{eqn:optuPDE}:
 \begin{align}
\fee^\alpha(\theta,t;\omega) := -\frac{1}{R}\ptheta h^\alpha (\theta,t;\omega,p)
 \label{eqn:optuPDEalpha}
\end{align}

The notation $\mfc(\theta,t;p)$ is used to stress that it is dependent upon the density $p$ (see formula~\eqref{eqn:barc}).   
As with the function $\mfc$,  we allow the  approximations $\{ \mfh^\alpha\} $ and hence also  elements of $\clH$  to depend upon $p$.   In implementation the density $p$ is replaced by the empirical distribution.

For any $\mfH^\alpha \in \clH$ its minimum over $u$ is denoted
\begin{equation*}
\Hopt^\alpha(\theta,t;\omega,p) := \min_{u\in\Re}      \mfH^\alpha(\theta,t,u;\omega,p)  \,  ,\quad \theta\in[0,2\pi)  \,  \  t\ge 0
\label{e:Hoptalpha_gen}
\end{equation*}
If the function class $\clH$  is ideal, so that $h$ in 
 \eqref{eq:p+h} coincides with $h^{\alpha^*}$ for some $\alpha^*\in\clA$,  then the HJB equation becomes  
\begin{equation}
\Hopt^{\alpha^*} (\theta,t;\omega,p)   =   \eta^*(\omega)
				  \,  ,\quad \theta\in[0,2\pi)  \,  \  t\ge 0 
\label{e:HJB_gen}
\end{equation}
In this case, $\mfH^{\alpha^*}$ is the Hamiltonian of optimal control,
and $\fee^{\alpha^*} $ in  
 \eqref{eqn:optuPDEalpha} is the optimal policy.

A traveling wave ansatz is introduced to justify our choices in the design of the
 function class $\clH$,    motivated by   discussion   in~\Sec{ssec:bif+phasetrans}.   
 In particular,  the form of the solution  \eqref{e:trwave}  is rationale for   the 
assumption  $\mfc(\theta,t; p)=\wavec(\theta-a t;p)$ and 
 $h^{\alpha}\,
(\theta,t;\omega,p)=\waveh^{\alpha}\, (\theta-a t;\omega,p)$
in   \eqref{e:mf_Qgen}, 
which leads to  $\mfH^\alpha(\theta,t,u;\omega,p)  =   
\waveH^\alpha(\theta-a t, u;\omega,p)$ with  
 \begin{equation*}
    \waveH^\alpha(\theta, u; \omega,p)  = \wavec(\theta; p) + \half Ru^2 +  \gen_{u}
    \waveh^{\alpha}\, (\theta;\omega,p),\quad \theta\in[0,2\pi),\;\;u\in\Re 
    \label{e:Qalpha}
\end{equation*}
with minimum  
$\waveHopt^\alpha(\theta;\omega,p) := \min_{u\in\Re} \waveH^\alpha(\theta,u; \omega,p)$.    
The HJB equation \eqref{e:HJB_gen}  becomes  
\[
\waveHopt^{\alpha^*}(\theta;\omega,p) =\eta^* (\omega)  ,\quad \theta\in[0,2\pi)
\]

However, it is unlikely that the HJB equation can be solved  exactly for any element of $\clH$.     Rather, a criterion of fit is proposed based on the  the point-wise Bellman error, 
\begin{equation}
    \PointwiseErr (\theta;\omega,p) :=  \waveHopt^\alpha(\theta;\omega, p)
    -  \eta^\alpha(\omega;p),  \quad \theta\in[0,2\pi)
      \label{e:PointwiseErrInf}
\end{equation}
where $\eta^\alpha(\omega;p) = 
\frac{1}{2\pi} \int_{0}^{2\pi} \waveHopt^\alpha(\theta;\omega,p) \ud\theta$.

\subsection{Kuramoto-inspired parametrization}

The numerical experiments illustrated in
\Fig{fig:ControlComp} show that the optimal control
law for the MF-HJB equation closely approximates the Kuramoto control law. 
This suggests the two-dimensional function class $\clH$ defined using  $\alpha = (A (\omega);  \phz(\omega)) \in\clA = \Re^2$ (for each fixed known $\omega\in\Omega$), and
\begin{empheq}[left=\text{(P1)}\;\quad,box=\fbox]{align*}
 \hspace{10pt}
    \waveh^{\alpha}(\theta;\omega) = - A(\omega)\cos \left(\theta - \phz(\omega) \right)     
 \hspace{10pt}
\end{empheq}
Consequently, 
for any $\alpha\in\clA$ the policy 
 \eqref{eqn:optuPDEalpha}
 becomes,   
\begin{align}
    \fee^{\alpha}(\theta, t;\omega) & 
    =  -\frac{1}{R} 
    A(\omega)  \sin (\theta - at - \phz(\omega) )
    \label{eqn:mfADPControl}
\end{align}

While (P1) is based on consideration of only the first harmonic, a more general Fourier parameterization is described in~\cite[Sec.~III.C]{huibing_TAC14}.

The optimal parameter $\alpha^* $  is defined as the solution to the 
 Galerkin relaxation of the HJB equation,  defined as the solution to   
\begin{align}
\begin{aligned} 
    \langle \PointwiseErr,\cos\rangle := \int_{0}^{2\pi} \wavePointwiseErr (\theta;\omega, p)   \cos(\theta)  \, \ud \theta &=  0
    \\ 
    \langle \PointwiseErr,\sin\rangle:= \int_{0}^{2\pi} \wavePointwiseErr (\theta;\omega, p)   \sin(\theta)  \, \ud \theta &=  0
\end{aligned}
\label{eqn:Galerkin}
\end{align}
The definition of the projected Bellman equation in Q-learning is defined by a similar Galerkin relaxation \cite[Sections 5.4 and 9.3]{CSRL}.

An explicit  solution to \eqref{eqn:Galerkin}
appears in~\cite[Thm.~3.2]{huibing_TAC14}. Observe that since $\PointwiseErr$ depends also upon $p$,  
the solution shares this dependence:
$\alpha^* = (A^*(\omega);\phz^*(\omega)) = (  A^*(\omega;p) ; \phz^*(\omega;p) )$.

Estimation of $\alpha^*$ is based on a mean-square loss function,      $
    \waveGalErr(\alpha ;\omega,p) :=  |\langle\wavePointwiseErr
    ,\cos\rangle|^2 + |\langle\wavePointwiseErr ,\sin \rangle|^2
$,
whose gradient may be expressed 
\begin{align*}
    \frac{\partial \waveGalErr}{\partial A} &= \frac{\pi^2}{2}\, (\popamp) \Big\{ 4A
    [(\omega-1)^2 + (\Noise)^2] + [ (\omega-1)\sin\phz - \Noise \cos\phz]
    \Big\} 
    \\ 
    \frac{\partial \waveGalErr}{\partial \phz} &= \frac{\pi^2}{2}\, (\popamp) \, A \big [(\omega - 1)
      \cos\phz + \Noise \sin\phz \big]
\end{align*}
where $\wavepProjC := \intsw p(\theta;\omega)\cos(\theta)g(\omega) \ud \theta
    \ud\omega$ and $\wavepProjS := \intsw p(\theta;\omega)\sin(\theta)g(\omega) \ud \theta
 \ud\omega$.
 
 Hence one could in principle estimate $\alpha^*$ through the gradient flow,
 \begin{align*}
\begin{aligned}
    \frac{\ud A }{\ud t} (t;\omega) &= - \epsilon \, (\popamp) \Big\{ 4A[(\omega -
    1)^2 + (\Noise)^2] + [ (\omega - 1) \sin\phz - \Noise \cos\phz] \Big\} \\
    \frac{\ud \phz}{\ud t} (t;\omega) &= -\epsilon \, (\popamp) \, A [(\omega -
    1) \cos\phz + \Noise \sin\phz]
\end{aligned}
\end{align*}
where $\epsilon>0$ is a small constant.    This motivates a practical approach described in the following subsection.

\subsection{Learning algorithm and its analysis}

In a finite-$N$ setting, we approximate $
    \popamp \approx    \Gamma^2_N(t)$,  where
\begin{equation*}
    \Gamma^2_N(t) =: 
\left(\frac{1}{N}\sum_{j =1}^N \sin(\theta_j(t))
    \right)^2 + \left(\frac{1}{N}\sum_{j=1}^N \cos(\theta_j(t))\right)^2 
\end{equation*}
and the learning rule for the $\ith$-oscillator is given by,
\begin{subequations}
\begin{align}
    \frac{\ud A_i}{\ud t}(t) &= -\epsilon \, \Gamma^2_N(t) \Big\{ 4A_i(t)[(\omega_i
    - 1)^2 + (\Noise)^2] + [(\omega_i - 1) \sin\phz_i(t) - \Noise \cos\phz_i(t)]
    \Big\} \\ 
    \frac{\ud \phz_i}{\ud t}(t) &= -\epsilon \, \Gamma^2_N(t) \,  A_i(t) [(\omega_i -
    1) \cos\phz_i(t) + \Noise \sin\phz_i(t)]
\end{align}
This is the desired learning rule for the amplitude and the phase parameters of the Kuramoto control law~\eqref{e:Kuramoto_control_parameter}.   

\label{eqn:StoGradAlgP2_finite}
\end{subequations}

The asymptotic behavior (as $t\rightarrow\infty$) of the learning
algorithm~\eqref{eqn:StoGradAlgP2_finite} for the finite-$N$
model are described in the following theorem, whose proof may be found in  \cite[Appendix VIII-D]{huibing_TAC14}:

\begin{theorem}[Theorem 4.1 in~\cite{huibing_TAC14}] \label{lem:eqmP2}
Consider the ODE~\eqref{eqn:StoGradAlgP2_finite} for updating
parameter vector
$\alpha_i(t)=(A_i(t),\phz_i(t))\in\Re\times[0,2\pi]$.  Suppose
$\Gamma^2_N(t) > 0$.  Then
\begin{romannum}
    \item The equilibria are given by $\PtwoEqm^{(1)} := (A_i^*,\phz_i^*)$, 
	$\PtwoEqm^{(2)} := (-A_i^*,\phz_i^*-\pi)$, 
	$\PtwoEqm^{(3)} := (0,\phz_i^*-\frac{\pi}{2})$ and
	$\PtwoEqm^{(4)} := (0,\phz_i^*+\frac{\pi}{2})$, 
	where $A_i^*=A^*(\omega_i),\,\phz_i^*=\phz^*(\omega_i)$ is the Galerkin solution.  
\item The equilibria $\PtwoEqm^{(3)}$ and $\PtwoEqm^{(4)}$ are both 
  unstable.  
\item The equilibria $\{\PtwoEqm^{(1)},\PtwoEqm^{(2)}\}$ constitute the
	global attracting set: For almost every initial condition
        $\alpha_i(0)=(A_i(0),\phz_i(0))\in\Re\times[0,2\pi]$,
        $\alpha_i(t)\rightarrow \PtwoEqm^{(1)}$ or
        $\alpha_i(t)\rightarrow \PtwoEqm^{(2)}$, as
        $t\rightarrow\infty$.   
    \end{romannum}
\end{theorem}

\begin{figure}
    \centering
    \includegraphics[width=.7\hsize]{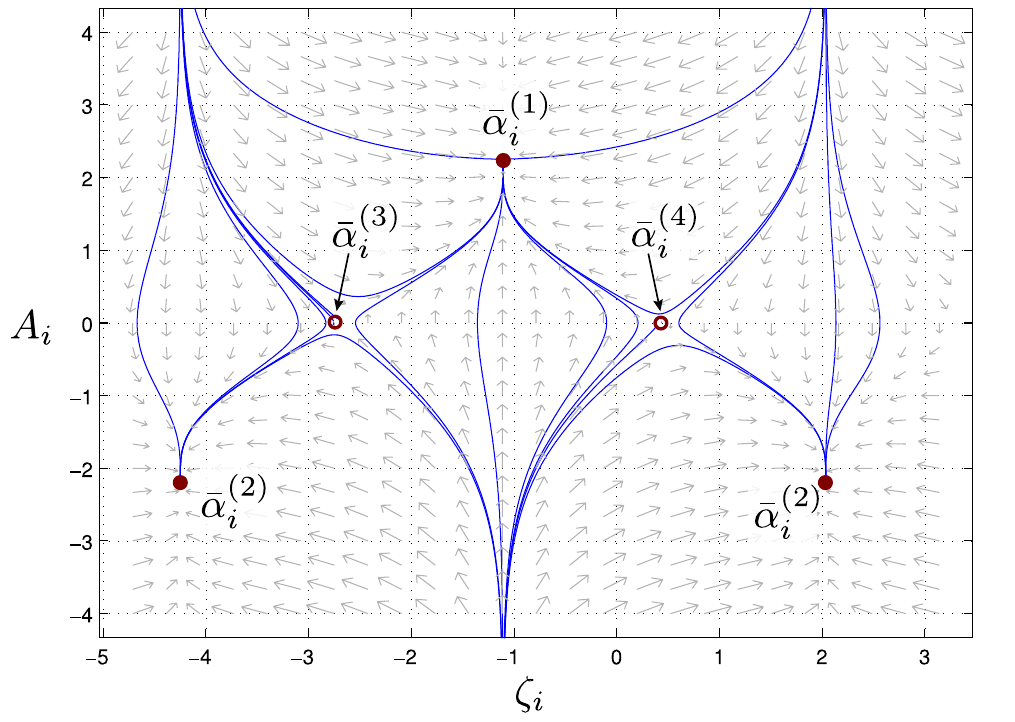}
    \vspace{-0.1in}
    \caption{Phase space plot for ODE~\eqref{eqn:StoGradAlgP2_finite}: The four equilibria are indicated
    by $\PtwoEqm^{(k)}$, $k = 1,\ldots,4$.
    }
    \label{fig:PhasePlot_StoGradAlgP2_mf}
    \vspace{-5pt}
\end{figure}

There are several remarks to be made concerning the assumptions and
conclusions of this theorem:

\begin{remark}
\label{rem:remark_learning}
  We make several remarks as follows:
\begin{romannum}
\item{\it Multiplicity of stable equilibria:} Either one of the two stable
  equilibria, $\{\PtwoEqm^{(1)}$ or $\PtwoEqm^{(2)}\}$, yields the same
  control law:
\[
u_i(t) = -\frac{A_i^*}{R} \frac{1}{N} \sum_{j =1}^N
    \sin(\theta_i(t) - \theta_j(t) - \phz_i^*)
\]
which coincides with the Galerkin control law~\eqref{eqn:mfADPControl} with $\alpha=\alpha^*$.  
\item{\it Phase space analysis:} The dynamics of the second order
  ODE~\eqref{eqn:StoGradAlgP2_finite} are visualized using the
  phase space plot depicted in \Fig{fig:PhasePlot_StoGradAlgP2_mf}. The parameter
  $\Gamma^2_N(t)$ determines the convergence time: For larger values
  of $\Gamma^2_N(t)$, it takes less time to converge to one of the two
  stable equilibria. 
\item{\it Justification of Assumption $\Gamma^2_N(t) > 0$:} The parameter
    $\Gamma^2_N(t)$ is a finite-$N$ approximation of the mean-field quantity
    $\popamp$. It is
    positive (resp.,~zero) when the population is in synchrony (resp.,~incoherence).  This shows that synchronization is useful for learning, in terms of quick convergence of parameters to their optimal values.

\item{\it  Relationship with Q-learning:}    The development here initially follows the theory leading to convex Q-learning in  \cite{mehmey09a,lumehmeyneu22,lumey23a},  
which proposes an alternative to the projected Bellman equation---the typical goal of Q-learning \cite{CSRL}.     In all of these 
approaches a version of the point-wise Bellman error  \eqref{e:PointwiseErrInf} arises, but it is a function of both state and action.   Consequently, the choice of input for training is a crucial part of any algorithm, as made abundantly clear in the recent stability analysis of Q-learning in \cite{mey24}.  
Special structure of the problem at hand leads to the simpler definition of Bellman error, and this is why there is no discussion of the exploration/exploitation tradeoffs in this section.
\end{romannum}
\end{remark}

The performance of the gradient descent algorithm in the finite-$N$ is explored  with the aid of simulations in the following
section.

\subsection{Numerics with the learning algorithm}

\begin{subequations}

The simulation results are described for a population of $N=200$ oscillators.  The
oscillator $i=1$ uses the learning update rule
\begin{align}
\ud\theta_1(t) &= \left(\omega_1 - \frac{A_1(t)}{R} 
    \frac{1}{N} \sum_{j=1}^N \sin(\theta_1(t) -\theta_j(t) - \phz_1(t))
    \right) \ud t  + \sigma \ud\xi_1(t)\\
    \frac{\ud A_1}{\ud t}(t) &= -\epsilon \, \Gamma^2_N(t) \Big\{ 4A_1(t)[(\omega_1
    - 1)^2 + (\Noise)^2]  + [(\omega_1 - 1) \sin\phz_1(t) - \Noise \cos\phz_1(t)]
    \Big\} \\ 
    \frac{\ud \zeta_1}{\ud t}(t) &= -\epsilon \, \Gamma^2_N(t) \,  A_1(t) [(\omega_1 -
    1) \cos\phz_1(t) + \Noise \sin\phz_1(t)]
\end{align} 
The remaining oscillators, $i=2,\hdots,N$, use the Kuramoto control law:
\begin{align*}
    \ud\theta_i(t) &= \left(\omega_i - \frac{\kappa}{N}\sum_{j=1}^N \sin(\theta_i(t) - \theta_j(t)) \right)\ud t +\sigma
    \ud \xi_i(t),\quad i=2,\hdots,N
\end{align*}
The frequency of the $i=1$ oscillator is taken to be $\omega_1 = 1.1$.  The remaining $N-1$ frequencies are sampled independently from the uniform
distribution on $\Omega=[0.9, 1.1]$, and $\sigma=0.1$.   

\label{e:finiteNalg}
\end{subequations}

\Fig{fig:LearningP2} depicts the results, trajectories for $A_1(t)$
and $\phz_1(t)$, for the two simulation cases:
\begin{romannum}
\item $\kappa = 0.01$ (population is in incoherence), and
\item $\kappa = 1$ (population is in synchrony).
\end{romannum}
Even though the incoherence and synchrony solution states are defined
for the limiting ($N\to\infty$) model, the terminology is useful here
to distinguish the population behavior in the two simulation cases.  

In the synchrony regime, the parameters converge \textit{quickly} to the
values $(A_i^*,\phz_i^*)$, consistent with the result in Thm.~\ref{lem:eqmP2}. Given the discussion
in Remark~\ref{rem:remark_learning}, it may be surprising to see that the
parameters converge even in the incoherence regime. The explanation is that
$\Gamma^2_N(t) = 0$ only in the limiting case ($N=\infty$).  For finite
$N$, it is small but non-zero.  So, the parameters converge to the optimal
values but relatively slowly.  The convergence is expected to get
progressively slower as $N$ increases.

\begin{figure}
\begin{center}
\includegraphics[width=0.75\hsize]{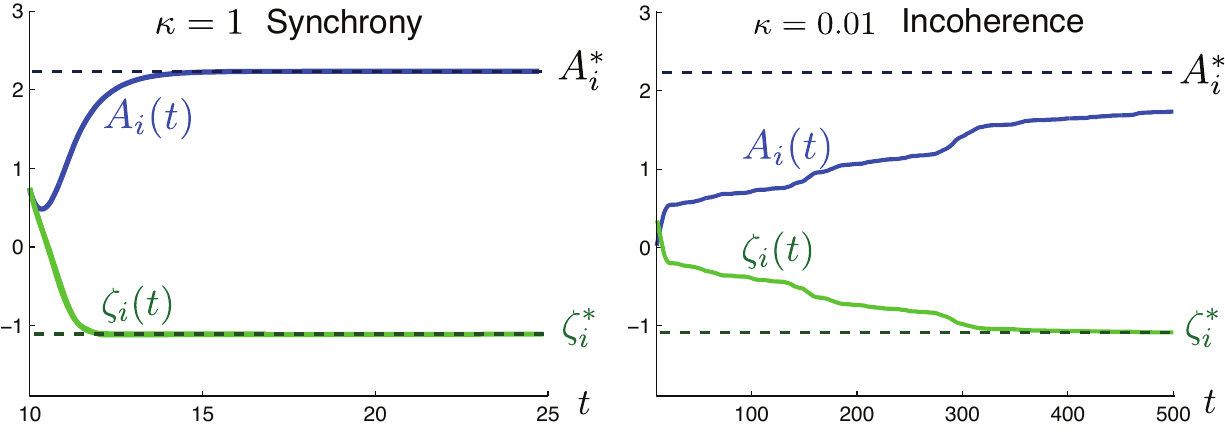}
\end{center}
    \caption{Simulation results with the learning algorithm: Synchronization
	of the population leads to a quicker convergence of control gain
	$A_i(t)$ and phase $\phz_i(t)$.%
    }
    \label{fig:LearningP2}
\end{figure}

\def\clZ{{\cal Z}}
\def\clL{{\cal L}}
\def\ast{{*}}
\def\fpfK{\textsf{K}}
\def\Inov{{\sf I}}
\def\T{[0,2\pi]}
\def\prechi{\psi}

\section{Coupled oscillator feedback particle filter}
\label{sec:FPF}

We turn next to approaches to state estimation for the oscillator model \eqref{e:osc}.  While we consider a single oscillator,  a mean field model emerges as a solution to the classical nonlinear filtering problem.

\begin{subequations}

We   introduce a scalar observation process $Z$   to obtain the   hidden Markov model 
    \begin{align}
        \ud \theta(t) &=\omega_0 \ud t + \sigma \ud \xi(t),\quad\text{mod}\;\;2\pi  \,, 
        \label{eqn:Signal_Process}
        \\
        \ud Z(t) &= h(\theta(t))\ud t + \ud W(t)
 \label{eqn:Obs_Process}
\end{align}
in which the interpretations of $\theta(t)\in [0,2\pi )$,  the nominal frequency  $\omega_0$,   and  $\xi:=\{\xi(t):t\geq 0\}$ are precisely as in 
\eqref{e:osc}.       In the observation model \eqref{eqn:Obs_Process} it is assumed that      
    $h:[0,2\pi)\mapsto \Re$ is a $C^1$ function, and 
    $W:=\{W(t):t\geq 0\}$ is a standard Wiener process.   The initial condition $\theta(0)$ is random, with   $C^2$  density   denoted   $p_0^\ast:[0,2\pi)\mapsto \Re^+$.   It is assumed that $\theta(0),\xi,W$ are mutually independent.

\end{subequations}  

The state  process $\{\theta(t):t\geq 0\}$ is partially  observed through the observation process   $\{Z(t):t\geq 0\}$.    
The objective of the filtering problem is to estimate the conditional density of $\theta(t)$ given the past history of observations (filtration) $\mathcal{Z}_t := \sigma(\{Z({\tau}):  \tau \le t\})$.    The model parameters 
  $\omega_0$, $h$  and $p_0^\ast$ are assumed to be known.

  The conditional density (also known as the posterior density) is denoted by $p^*(\theta,t)$ at time $t$,    defined so that, for any measurable set $A\subset [0,2\pi]$,
\begin{equation}
    \int_{\theta \in A} p^*(\theta,t)\, \ud \theta   = \Prob[ \theta(t) \in A\mid \clZ_t ],\quad t\geq 0
    \label{e:pXDef}
\end{equation}
The evolution of $p^*(\theta,t)$ is described by the Kushner-Stratonovich (K-S) stochastic PDE:
\begin{align*}
    \ud p^\ast(\theta,t) &= ({\cal D}^\dagger p^\ast)(\theta,t) \ud t + ( h(\theta)-\bar{h}(t) ) (\ud Z(t) - \bar{h}(t)
    \ud t)p^\ast(\theta,t), \quad \theta\in[0,2\pi),\;t\geq 0\\
      p^\ast(\theta,0) &= p_0^\ast(\theta),\quad \theta\in[0,2\pi)
\end{align*}
where  $ \bar{h}(t) = \int_{0}^{2\pi} h(\theta) p^*(\theta,t) \ud \theta$,  and 
\begin{equation*}
 {\cal D}^\dagger p^\ast(\theta,t) := - \omega_0 \frac{\partial p^\ast}{\partial \theta} (\theta,t)+
 \frac{\sigma^2}{2} \frac{\partial^2 p^\ast}{\partial \theta^2}(\theta,t), \quad \theta\in[0,2\pi),\;t\geq 0
 \label{eqn:dagger}
\end{equation*}
Here, $\frac{\partial }{\partial \theta}$ and $\frac{\partial^2
}{\partial \theta^2}$  denote the first and second partial derivatives with respect to $\theta$, respectively. Because the evolution of the posterior is given by a PDE, the filter is said to be infinite-dimensional.

\subsection{Coupled oscillator feedback particle filter}

Expressed in its Stratonovich form, the coupled oscillator feedback particle filter (FPF) is given by  
\begin{align}
    \ud \theta_i(t) = \omega_i \ud t + \sigma \ud \xi_i(t) + \fpfK(\theta_i(t),t) \circ \left(\ud Z_t - \frac{1}{2}\left(h(\theta_i(t)) +
    \hat{h}(t)\right)\ud t\right),\;\;\text{mod}\;2\pi \,, 
  \label{eqn:prelim_fpf_Strat}
\end{align}
where $\omega_i=\omega_0$,  $\hat{h}(t) := \int_{0}^{2\pi} h(\theta) p(\theta,t) \ud \theta$,
and
and  $\{ \theta_i(0) \}$ are i.i.d..

The gain function  $ \fpfK(\varble ,t) $ is in fact a stochastic process adapted to the observations,  defined as a function of the conditional density of a single particle.

Denote by $p(\varble,t)$   the   conditional  density of $\theta_i(t)$ given $\clZ_t$:      in analogy with   \eqref{e:pXDef},
 \begin{equation}
    \int_{\theta \in A} p(\theta,t)\, \ud \theta   = \Prob[ \theta_i(t) \in A\mid \clZ_t ],\quad t\geq 0
    \label{e:pXDefi}
\end{equation}
The density $p$  does not depend upon $i$ under the i.i.d.\ assumption for $\{ \theta_i(0) \}$.     
\begin{subequations}
We then let  $\phi$ denote the solution to   Poisson's equation  
\begin{equation}
    \label{eqn:EL_phi_prelim}
        - \frac{\partial}{\partial \theta}  \left[ p(\theta,t)
          \frac{\partial \phi}{\partial \theta} (\theta,t) \right] 
          	 =    \left[  h(\theta)-\hat{h}(t) \right]  p(\theta,t), \quad \theta\in[0,2\pi),\;t\geq 0
\end{equation}
 normalized with $        \int_{0}^{2\pi} \phi(\theta,t) \ud \theta  = 0$.    The gain function $\fpfK$ is its gradient:
\begin{equation}
    \fpfK  (\theta,t) = \frac{\partial\phi}{\partial \theta} (\theta,t)
    \label{eqn:gradient_gain_fn_prelim}
\end{equation}

\label{eq:PoissonEqn}
\end{subequations}

The evolution of $p(\varble,t)$ is given by a Fokker-Planck-Kolmogorov (FPK) equation~\cite[Prop.~3.1]{taoyang_TAC12}.  Its significance to the filtering problem comes from the following theorem whose proof appears in~\cite[Appendix F]{taoyang_TAC12}.

\begin{theorem}
    \label{thm:kushner}
 Suppose that the gain
    function is obtained according
    to~\eqref{eq:PoissonEqn},
 and  $\{ \theta_i(0) \}$ are sampled i.i.d.\ from  
    $p^*_0$.   
 Then,    
 \begin{equation*}
	p(\theta,t) = p^*(\theta,t),\quad \theta\in[0,2\pi),\;t>0
    \end{equation*}
\end{theorem}

\begin{remark}
  \Theorem{thm:kushner} is based on the \textit{ideal setting} in which the gain $\fpfK(\theta_i(t),t)$ and the prediction $\hat{h}(t)$ are obtained as a function of the conditional density $p$ for $\theta_i(t)$  (equivalently, obtained from the conditional density $p^*$    for $\theta(t)$  in view of the theorem).      
  In practice, the algorithm is applied with $p$ replaced by the empirical distribution of the $N$ particles.  For example, with $N$ particles, $\hat{h}(t) \approx \hat{h}^{(N)}(t):= \frac{1}{N} \sum_{i=1}^N h(\theta_i(t))$.  For additional details on the FPF algorithm, see our recent review papers~\cite{taghvaei2023survey,taghvaei2021optimal}.
\end{remark}

\subsection{Galerkin algorithm for solving the BVP}
\label{sec:COPF}
 
The main issue in implementing the FPF algorithm is to obtain the
solution to the BVP~\eqref{eqn:EL_phi_prelim} at each time step.  In
this section, a Galerkin algorithm for the same is described.  

\paragraph{Weak formulation.}
The notation $L^2([0,2\pi];p)$ is used to denote the Hilbert
space of $2\pi$-periodic functions on $[0,2\pi]$ that are
square-integrable with respect to density $p$;
$H^k(\T;p)$ is used to denote the Hilbert space of functions
whose first $k$-derivatives (defined in the weak sense) are in
$L^2(\T;p)$. Denote
\[
H_0^1(\T;p) := \left\{ \phi\in H^1 \,\Big|\, \int
\phi(\theta) \ud \theta =0 \right\}
\]
A function $\phi \in H_0^1(\T;p)$ is said to be a weak solution of the
BVP~\eqref{eqn:EL_phi_prelim} if
\begin{align*}
    \int \frac{\partial \phi}{\partial \theta} (\theta,t)
    \frac{\ud \psi}{\ud \theta}(\theta) \; p(\theta,t) \ud
    \theta 
     = \int  (h(\theta)-\hat{h}(t)) \psi(\theta) \; p(\theta,t) \ud
    \theta
\end{align*}
for all $\psi \in H^1(\T;p)$.  
Denoting $\Expect [\cdot]:= \int \cdot \; p(\theta,t) \ud \theta$, the weak
form can be expressed as follows:
\begin{equation}
    \Expect [\phi' \psi' ] = \Expect[  (h -\hat{h}(t))
    \psi],\quad\forall\,\,\psi\in H^1(\T;p)
\label{eqn:EL_phi_expect}
\end{equation}
where $\phi':=\frac{\partial \phi}{\partial \theta}$ and $\psi':=\frac{\ud \psi}{\ud \theta}$.

\paragraph{Finite-dimensional approximation.} One begins by pre-selecting a finite set of basis functions $\{\prechi_l\}_{l=1}^L$ where $\prechi_l:[0,2\pi)\mapsto \Re$ for $l=1,2,\hdots,L$. In terms of these basis functions, the gain function $\phi(\theta,t)$ is approximated as,
\begin{equation}
    \phi(\theta,t) = \sum_{l=1}^L \kappa_l(t) \prechi_l(\theta)\nonumber
\end{equation}
To compute the coefficients $\{\kappa_l(t):l=1,2,\hdots,L\}$, the Galerkin approximation of the weak form~\eqref{eqn:EL_phi_expect} is as follows:
\begin{equation*}
    \sum_{l=1}^L \kappa_l(t) \Expect[ \psi'_l \; \prechi_k' ] =
    \Expect [  (h-\hat{h}(t)) \psi_k ],\quad\forall\,\,k=1,2,\hdots,L
\end{equation*}
Denoting $[A]_{kl} = \Expect[ \prechi'_l \; \prechi'_k ]$, $b_k
= \Expect [  (h-\hat{h}(t)) \prechi_k ]$,
$\kappa=(\kappa_1(t),\kappa_2(t),\hdots,\kappa_L(t))^T$, the matrix form for the same is
\begin{equation*}
    A\kappa(t) = b
\end{equation*}
In a numerical implementation, the matrix $A$ and vector $b$ are approximated empirically as 
\begin{align*}
    [A]_{kl} & = \Expect[ \prechi'_l \;  \prechi'_k ] \approx
    \frac{1}{N} \sum_{i=1}^N \prechi'_l (\theta_i(t)) \; 
    \prechi'_k (\theta_i(t)), 
    \\
    b_k & = \Expect [  (h-\hat{h}(t)) \psi_k ]  \approx \frac{1}{N} \sum_{i=1}^N
    (h(\theta_i(t)) - \hat{h}^{(N)}(t)) \psi_k(\theta_i(t)), 
\end{align*}
Finally, using~\eqref{eqn:gradient_gain_fn_prelim}, the gain function $\fpfK$ is obtained as,
\begin{equation*}
    \fpfK(\theta,t) = \frac{\partial \phi}{\partial \theta} (\theta,t)=
    \sum_{l=1}^L \kappa_l(t) 
    \prechi'_l (\theta)
\end{equation*}

\medskip

\begin{example}\label{ex:ex_gain_cos}
\textup{
A natural choice of basis functions $S =\{\cos(\theta),\;\sin(\theta)\}$. Then 
\begin{equation*}
    \fpfK(\theta,t) = - \kappa_1(t) \sin(\theta) + \kappa_2(t) \cos(\theta)
\end{equation*}
where $[\kappa_1(t), \kappa_2(t)]$ are obtained at each time by
solving the linear matrix equation: 
\begin{align*}
    \left[ \begin{array}{cc} 
        \frac{1}{N} \sum_{j} \sin^2(\theta_j(t)) & - \frac{1}{N} \sum_j
        \sin(\theta_j(t)) \cos(\theta_j(t)) 
        \\
        - \frac{1}{N} \sum_j \sin(\theta_j(t)) \cos(\theta_j(t)) & \frac{1}{N} \sum_j
        \cos^2(\theta_j(t)) 
    \end{array} \right]
    \left[ \begin{array}{c} 
        \kappa_1(t) 
        \\
        \kappa_2(t) 
    \end{array}\right] \nonumber
     \\
     = { \left[ \begin{array}{c} 
        \frac{1}{N} \sum_j \left( h(\theta_j(t)) - \frac{1}{N} \sum_{i}
        h(\theta_i(t)) \right) \cos(\theta_j(t)) 
        \\
        \frac{1}{N} \sum_j \left( h(\theta_j(t)) - \frac{1}{N} \sum_{i}
        h(\theta_i(t)) \right) \sin(\theta_j(t)) 
    \end{array}\right]}
\end{align*}
With $h(\theta)=\cos(\theta)$, the evaluation of the coefficients requires
computation of empirical averages $\frac{1}{N} \sum_i \sin(\theta_i(t))$,
$\frac{1}{N} \sum_i \cos(\theta_i(t))$, $\frac{1}{N} \sum_i
\sin(2\theta_i(t))$ and $\frac{1}{N} \sum_i \cos(2\theta_i(t))$ at each
time step.}
\end{example}

\subsection{Numerics}

We describe the results of a numerical experiments for the noise-free phase model $\theta(t) = \omega_0 t$ with $\omega_0=1$, and noisy observations using $h(\theta)=\cos(\theta)$. The coupled oscillator FPF is simulated with $N=1000$ particles.  For
$i=1,\hdots,N$:
\begin{equation*}
    \ud \theta_i(t) = \omega_i \ud t + \sigma_B \ud \xi_i(t) + \langle \kappa(t),
    \prechi'(\theta_i(t))\rangle \circ \left( \ud Z(t) - \frac{1}{2}    (h(\theta_i(t)) +
    \hat{h}^{(N)}(t)) \ud t \right) ,\quad\text{mod}\;\;2\pi
\end{equation*}
where the initial conditions $\{\theta_i(0):i=1,2,\hdots,N\}$ are sampled i.i.d from a uniform distribution on $[0,2\pi)$. The filter frequencies $\omega_i$ are sampled from a uniform distribution on $[1-\gamma,1+\gamma]$ with $\gamma=0.5$ and a small process noise is included by chosing $\sigma_B=0.1$.  The gain $\kappa(t) = [\kappa_1(t),\kappa_2(t)]$ is obtained at each discrete time-step $t$ by solving the $2\times 2$ linear equation described in \Ex{ex:ex_gain_cos}. The notation $\langle \kappa(t),
    \prechi'(\theta_i(t))\rangle= -\kappa_1(t)\sin(\theta_i(t)) + \kappa_2(t)\cos(\theta_i(t))$. The filter is simulated using Euler discretization with time-step $\Delta t = 0.01$.

\begin{figure*}[h]
\centering
        \includegraphics[width=\hsize]{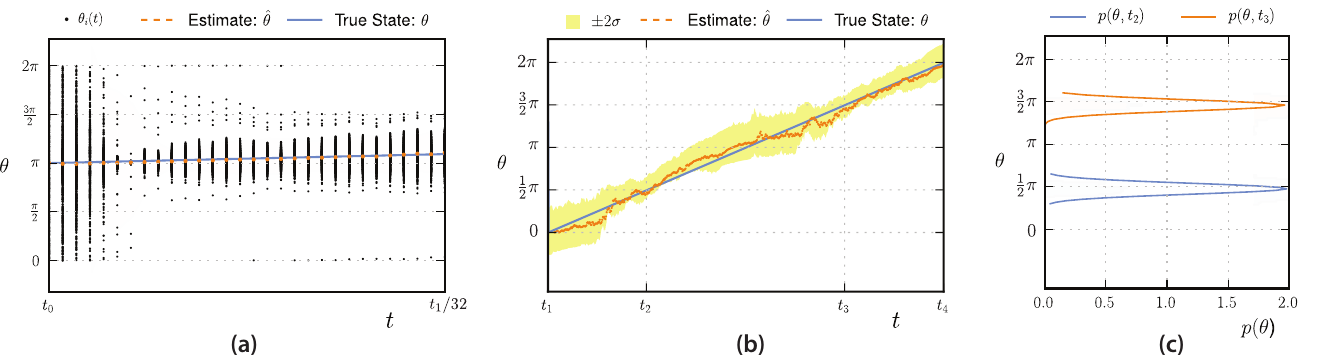}
    \caption{Summary of numerical results for example 1:
    (a) Initial transients as particles converge starting
    from a uniform initial distribution; (b) Comparison
    of estimated mean $\{\hat{\theta}^{(N)}(t)\}$ with the true state
    $\{\theta(t)\}$, where the shaded region show $\pm$ two standard
    deviations; and (c) representative empirical
    distributions at two times, $t_2$ and $t_3$.}
    \label{fig:single_hopf}
\end{figure*}

\Fig{fig:single_hopf} depicts numerical results after the initial transients have converged. A single cycle is depicted in part~(b): the figure compares the sample path of the actual state $\theta(t)$ (as a solid line) with the estimated mean $\hat{\theta}^{(N)}(t)$ (as a dashed line). The shaded area indicates $\pm$ two standard deviation bounds.  The density at two time instants is depicted in part~(c).  These results show that the filter is able to track the phase accurately.  

Note that the exact filter requires $\sigma_B=0$ and $\omega_i=\omega_0=1$.  A small noise is included and the frequencies are picked from a distribution to (i) highlight the robustness of the proposed algorithm and (ii) because such a model is more realistic from an implementation standpoint.  The study also invites an investigation of the synchronization phenomena in coupled oscillator FPF.  This is the subject of future work.

\section{Conclusions and directions for future work}
\label{sec:conc}

The field of MFGs combines  concepts and techniques from nonlinear dynamical systems, game theory, and statistical mechanics.   While initially formulated in \cite{huang06large,lasry07mean} as an approach to distributed control, we have surveyed here how MGF techniques also provide a framework for state estimation.  See the recent survey \cite{taghvaei2023survey}  
for more on
future directions for research in the context of state estimation.  Prior work \cite{mehmey09a,busmeycam23} along with the theory from \cite{huibing_TAC14} surveyed here demonstrates how similar techniques can be used as a part of algorithms for  distributed learning.

Our focus on coupled oscillators comes from the relevance of such models to neuroscience, specifically, to the various hypotheses around the role of synchronization and neural rhythms in cortical circuits.  In this regard, it will be useful to extend the mathematics to settings of multiple populations with an architecture inspired by cortex.  It will also be useful to relate some of the model predictions to the in vivo neural recordings of perceptual learning and inference tasks.

\wham{Acknowledgement}

Over the years, a number of students and colleagues have contributed to this research: Rohan Arora, Heng-Sheng Chang, Shane Ghiotto, Jin-Won Kim, Yagiz Olmez, Uday Shanbhag, Amirhossein Taghvaei, Adam Tilton, Tixian Wang, Tao Yang, and Huibing Yin. These contributions and collaborations are gratefully acknowledged.

\newpage

\bibliographystyle{abbrv}


\def\cprime{$'$}\def\cprime{$'$}

\end{document}